\documentclass{amsart}

\usepackage{amscd}
\usepackage{amsmath,amssymb,amsfonts}
\usepackage{xypic}
\usepackage[mathcal]{euscript}
\usepackage[cmtip,all,2cell]{xy}
\usepackage{url}
\usepackage{latexsym}
\usepackage{amsthm}
\usepackage[latin1]{inputenc}
\usepackage{multicol}
\usepackage{enumitem}
\usepackage{hyperref}

\usepackage{tikz}

\UseAllTwocells

\newtheorem{theorem}{Theorem}[section]
\newtheorem{lemma}[theorem]{Lemma}
\newtheorem{proposition}[theorem]{Proposition}

\newtheorem{corollary}[theorem]{Corollary}
\newtheorem{definition}[theorem]{Definition}

\theoremstyle{definition}
\newtheorem{example}[theorem]{Example}

\newtheorem{remark}[theorem]{Remark}

\newcommand{\ol}[1]{\overline{#1}}
\newcommand{\cat}[1]{\mathbf{#1}}

\newcommand{\sSet}[0]{\cat{sSet}}
\newcommand{\dSet}[0]{\cat{dSet}}
\newcommand{\colim} {\operatornamewithlimits{\mathrm{colim}}}

\newcommand{\dau}[0]{\partial}

\newcommand{\mm}[1]{\mathrm{#1}}

\newcommand{\Hom}[0]{\mm{Hom}}
\newcommand{\rt}{\rightarrow}

\numberwithin{equation}{section} 

\begin{document}
\title{Minimal fibrations of dendroidal sets}
\author{Ieke Moerdijk \and Joost Nuiten}
\address{Radboud Universiteit Nijmegen, Institute for Mathematics, Astrophysics and Particle Physics, Heyendaalseweg 135, 6525 AJ Nijmegen, The Netherlands}
\email{i.moerdijk@math.ru.nl \\ j.nuiten@math.ru.nl}

\maketitle

\vspace{-25pt}
\begin{abstract}
We prove the existence of minimal models for fibrations between dendroidal sets in the model structure
for $\infty$-operads, as well as in the covariant model structure for algebras and in the stable one for
connective spectra. In an appendix, we explain how our arguments can be used to extend the results of \cite{cis14}, giving the existence of minimal fibrations in model categories of presheaves over generalised Reedy categories of a rather common type. Besides some applications to the theory of algebras over $\infty$-operads, we also prove a gluing result for parametrized connective spectra (or $\Gamma$-spaces).
\end{abstract}

\section{Introduction}
A classical fact in the homotopy theory of simplicial sets -- tracing back to J.~ C.~ Moore's lecture notes from 1955-56 -- says that any Kan fibration between simplicial sets is homotopy equivalent to a fiber bundle \cite{bar59, gab67, may67}. This is proven by deforming a fibration onto a so-called minimal fibration, a Kan fibration whose only self-homotopy equivalences are isomorphisms. Such minimal fibrations provide very rigid models for maps between simplicial sets -- in particular, they are all fiber bundles -- which are especially suitable for gluing constructions.

Essentially the same method allows one to construct minimal categorical fibrations between $\infty$-categories as well (cf.~ \cite{joy08, lur09}). In fact, these two constructions are particular cases of a general statement on the existence of minimal fibrations in certain model structures on presheaves over Reedy categories, proved by Cisinski in \cite{cis14}. The case of dendroidal sets is not covered by this result however, due to the presence of nontrivial automorphisms in the base category $\Omega$.

The aim of this note is to show that the basic theory of minimal fibrations extends naturally to the setting of dendroidal sets. We say that an operadic fibration $p\colon Y\rt X$ of dendroidal sets (cf.~ \cite{cis09}) is \emph{minimal} if all weak equivalences over $X$
$$\xymatrix@C=1.5pc@R=1.8pc{
Y\ar@{->>}[rd]_p\ar[rr]^-\sim & & Y\ar@{->>}[ld]^p\\
& X &
}$$
are isomorphisms. This terminology is justified by the fact that any trivial cofibration from another fibration into the fibration $p$
$$\xymatrix@C=1.5pc@R=1.8pc{
\tilde{Y}\ar@{->>}[rd]\ar@{>->}[rr]^\sim & & Y\ar@{->>}[ld]^p\\
& X
}$$
is an isomorphism. Indeed, any such trivial cofibration $i$ admits a retraction $r$ with the property that the composite $ir\colon Y\rt Y$ is a self-weak equivalence of $Y$ over $X$ and therefore an isomorphism.

The presence of nontrivial automorphisms in $\Omega$ makes the discussion of minimal fibrations a bit more delicate. For instance, the pullback of a minimal fibration need no longer be minimal again (see Remark \ref{rem:minimalfibrationsarenotpullbackstable} below). Our main result asserts that an operadic fibration can nonetheless be retracted onto a weakly equivalent minimal fibration, although Quillen's argument \cite{qui68} showing that this retraction is a trivial fibration no longer applies in general:
\begin{theorem}\label{thm:existenceofminimalmodels}
Let $p\colon Y\rt X$ be an operadic fibration between dendroidal sets with normal domain. Then the following holds:
\begin{itemize}
 \item[(a)] $p$ admits a minimal fibration $M\rt X$ as a fiberwise strong deformation retract.
 \item[(b)] the retraction $r\colon Y\rt M$ is a trivial fibration of dendroidal sets when the codomain $X$ is normal.
\end{itemize}
\end{theorem}
The proof of this theorem appears in Section \ref{sec:proof} and proceeds by induction along the skeletal filtration of the domain $Y$, analogously to the classical case of simplicial sets.

One may also find minimal models for the fibrations in the covariant and stable model structures \cite{bas12} on dendroidal sets, or any other left Bousfield localization of the operadic model structure on dendroidal sets. Indeed, any fibration $p\colon Y\rt X$ in a left Bousfield localization of the operadic model structure is in particular an operadic fibration. The associated minimal operadic fibration is a retract of $p$ and therefore a local fibration. It is minimal in the localized model structure since the local weak equivalences between local fibrations over $X$ coincide with the operadic weak equivalences.

The same argument shows that any left fibration $Y\rt X$ of dendroidal sets admits a minimal model. Such a left fibration is not quite a fibration in a certain model category, but instead it defines a fibrant object in the covariant model structure on the over-category $\cat{dSet}/X$. This model structure has been constructed in \cite{heu12}, where it is also shown to be Quillen equivalent to the model category of algebras (in $\sSet$) over the simplicial operad associated to $X$.

A map $f\colon X\rt X'$ between dendroidal sets induces a Quillen pair between the covariant model structures
$$\xymatrix{
f_!\colon \cat{dSet}{/X} \ar@<1ex>[r] & \cat{dSet}{/X'}\colon f^*\ar@<1ex>[l]
}$$
which is a Quillen equivalence whenever $f$ is an operadic weak equivalence (\cite{heu12}, Proposition 2.4). As such, one obtains a (relative) functor
$$\entrymodifiers={+!!<0pt,\fontdimen22\textfont2>}\xymatrix@1{
\mm{Alg}\colon \cat{dSet}^{\text{op}}\ar[r] & \cat{ModelCat}^{\text{R}} & X\ar@{|->}[r] & \left(\dSet/X\right)^{\text{cov}}
}$$
taking values in model categories with right Quillen functors between them. We use the theory of minimal fibrations to prove the following
\begin{proposition}\label{prop:descentforleftfibrations}
The functor $\mm{Alg}$ preserves homotopy pullbacks. More precisely, for any diagram of dendroidal sets $X_1\leftarrow X_0\rt X_2$ in which both arrows are cofibrations, the natural adjoint pair
$$\xymatrix{
\colim\colon \cat{dSet}{/X_1}\times^h_{\cat{dSet}{/X_0}} \cat{dSet}{/X_2} \ar@<1ex>[r] & \cat{dSet}{/X_1\cup_{X_0} X_2}\ar@<1ex>[l]\colon \mm{pullback}
}$$
establishes a Quillen equivalence between the homotopy pullback model structure and the covariant model structure on $\cat{dSet}{/X_1\cup_{X_0} X_2}$.
\end{proposition}
Informally, this proposition asserts that algebras over a homotopy pushout of $\infty$-operads can equivalently be described as (homotopy) matching triples of algebras over the individual pieces of the homotopy pushout. We will explain and prove Proposition \ref{prop:descentforleftfibrations} in Section \ref{sec:applications}, where we also use the theory of minimal left fibrations to give an elementary proof of a result from \cite{heu12} about weak equivalences between left fibrations.

In the Appendix, we briefly discuss how the arguments of the present paper yield a general existence theorem for minimal fibrations over a large class of so-called \emph{generalised} Reedy categories, providing a common generalization of Cisinski's result for strict Reedy categories \cite{cis14} and ours for dendroidal sets. As an application of this extended result, we have included a gluing result for parametrized connective spectra, analogous to Proposition \ref{prop:descentforleftfibrations}.

\section{Preliminaries on dendroidal sets}\label{sec:preliminariesondendroidalsets}
Recall that the category of dendroidal sets is the category of set-valued presheaves on the category $\Omega$ of finite rooted trees \cite{moe10, moe07, moe09}. The category $\Omega$ comes equipped with two wide subcategories $\Omega^+$ (resp. $\Omega^-$), whose arrows are those maps of trees that induce an injection (resp. surjection) on edges. We will call maps in $\Omega^+$ face maps and maps in $\Omega^-$ degeneracy maps. The intersection $\Omega^+\cap \Omega^-$ consists of the isomorphisms in $\Omega$ and every map in $\Omega$ factors essentially uniquely as a degeneracy map, followed by a face map.

For any finite rooted tree $T$, the degree of $T$ is given by the number of vertices of $T$. It is immediate that non-invertible arrows in $\Omega^+$ (resp.~ $\Omega^-$) raise (resp.~ lower) the degree. Altogether, this gives the category $\Omega$ the structure of a (generalised) \emph{Reedy category}.

Any degeneracy map can be realized as the composition of isomorphisms and elementary degeneracy maps, i.~e.~ maps $\sigma_v\colon T\rt T\setminus v$ obtained by picking a vertex $v$ of $T$ with a single input, removing that vertex and identifying the incoming and outgoing edges. An elementary degeneracy $T\rt T\setminus v$ admits precisely two sections, obtained by choosing an edge above or below the vertex $v$ and considering the face map induced by contracting this edge.
\begin{lemma}\label{lem:degeneraciesdeterminedbysections}
Any degeneracy map is a split epimorphism and two degeneracy maps $\sigma, \tau\colon T\rt S$ are the same if they have the same set of sections.
\end{lemma}
\begin{proof}
Recall that a map of trees is completely determined by its effect on the set of edges, so that the sections of a degeneracy $\sigma\colon T\rt S$ form a subset of the set of sections of the induced surjection $\sigma_*\colon \mm{Edge}(T)\rt \mm{Edge}(S)$. On the other hand, any section $i$ of $\sigma_*$ is induced by the face map $\delta\colon S\rt T$ that contracts all edges of $T$ which are not contained in the image of $i$. This face map is a section since $\sigma\delta$ induces the identity map on colours. It follows that sections of a degeneracy $\sigma$ correspond bijectively to sections of the associated surjection between sets of edges. The second assertion now follows from the fact that surjections of sets are uniquely determined by their sets of sections, while the first assertion is obvious.
\end{proof}
\begin{lemma}[\cite{moe10}, 3.1.6]\label{lem:pushoutofsplitepis}
Any pair of degeneracies $\sigma\colon S\rt S'$, $\tau\colon S\rt T$ fits into an absolute pushout square
\begin{equation}\label{diag:absolutepushoutofsplitepis}\vcenter{\xymatrix{
S\ar[r]^{\sigma}\ar[d]_{\tau} & S'\ar[d]^{\tau'}\\
T\ar[r]_{\sigma'} & T'
}}\end{equation}
in which $\sigma'$ and $\tau'$ are degeneracies as well.
\end{lemma}
\begin{proof}
Since absolute pushout squares can be pasted, it suffices to check this when $\sigma=\sigma_v\colon S\rt S\setminus v$ and $\tau=\sigma_w\colon S\rt S\setminus w$. In this case, one can easily check that the required pushout square \eqref{diag:absolutepushoutofsplitepis} can be produced by taking $T'=S\setminus \{v, w\}$ and $\sigma'$ (resp.~ $\tau'$) the elementary degeneracy removing the vertex $v$ (resp. $w$) and identifying the ingoing and outgoing edge. In the case where $v=w$, the maps $\sigma'$ and $\tau'$ are simply the identity maps.

To see that the resulting pushout square is an \emph{absolute} pushout square, it suffices to find sections $\alpha$ of $\sigma$ and $\alpha'$ of $\sigma'$ which are compatible in the sense that $\tau\alpha=\alpha'\tau'$ (see e.~g.~ \cite{ber13}). When the vertices $v$ and $w$ are the same, one can just pick any section of $\sigma=\sigma_v$ and take the identity section of $\sigma'=\mm{id}$. If $v$ is different from $w$ and $v$ is not connected to $w$ by a single edge, one can take both $\alpha$ and $\alpha'$ to be the face map contracting the edge below the vertex $v$ (seen as a vertex in $S$, resp. $S\setminus w$).

We are left with the case that the vertices $v$ and $w$ are connected by a single edge. If $v$ is the vertex directly above $w$, compatible sections are provided by letting $\alpha$ and $\alpha'$ be the face maps contracting the edge above $v$ (again seen as a vertex in $S$, resp. $S\setminus w$). If $v$ is the vertex direcly under $w$, one can take $\alpha$ and $\alpha'$ to be the face maps contracting the edge below $v$.
\end{proof}
We identify elements of a dendroidal set $X$ with maps $x\colon \Omega[T]\rt X$, where $\Omega[T]$ is the presheaf represented by the tree $T$. An element $x\colon \Omega[T]\rt X$ is called degenerate if it factors as $\Omega[T]\rt \Omega[S]\rt X$, where $T\rt S$ is a degeneracy. It follows easily from Lemma \ref{lem:pushoutofsplitepis} that any element of a dendroidal set decomposes essentially uniquely as a degeneracy of a nondegenerate element (see e.~g.~ Prop.~ 6.7 in \cite{ber08}).

For every tree $T$, there is an action of the automorphism group $\mm{Aut}(T)$ on the set of nondegenerate elements $\Omega[T]\rt X$. If $x\colon \Omega[T]\rt X$ is an  element of $X$, define its automorphism group $\mm{Aut}(x)\subseteq\mm{Aut}(T)$ to be the isotropy group of the element $x$ under this action. A map of dendroidal sets $f\colon X\rt Y$ induces a (necessarily injective) map $\mm{Aut}(x)\rt \mm{Aut}(fx)$.

A monomorphism $i\colon A\rt B$ between dendroidal sets is called \emph{normal} if a nondegenerate element of $B$ has a trivial automorphism group whenever it does not factor through $i$. In other words, $\mm{Aut}(T)$ acts freely on the set of nondegenerate elements in $B(T)\setminus A(T)$. A dendroidal set $X$ is called normal if the map $\emptyset\rt X$ is a normal monomorphism. 
\begin{remark}
In fact, for a normal monomorphism $i\colon A\rt B$ the group $\mm{Aut}(T)$ acts freely on the set of \emph{all} elements $\Omega[T]\rt B$ that do not factor through $i$ (Prop.~ 1.5 in \cite{cis09}). An easy consequence of this is the fact that any monomorphism over a normal dendroidal set is a normal monomorphism. 
\end{remark}

\subsection*{Skeletal filtration}
Let $t_n\colon \Omega_{\leq n}\rt \Omega$ be the inclusion of the full subcategory of $\Omega$ on the objects of degree $\leq n$. The $n$-skeleton of a dendroidal set $X$ is given by $X^{(n)}:=t_{n!}t_{n}^*X$. The skeleta of $X$ fit into a natural skeletal filtration
\begin{equation}\label{diag:skeletalfiltration}\xymatrix{
\emptyset= X^{(-1)}\ar[r] & X^{(0)}\ar[r] & X^{(1)}\ar[r] & \cdots \ar[r] & X.
}\end{equation}
Because every element of a dendroidal set $X$ is a degeneracy of a nondegenerate element in an essentially unique way, the maps in the skeletal filtration \eqref{diag:skeletalfiltration} are all monomorphisms and the colimit of this sequence of inclusions is the original dendroidal set $X$. Indeed, $X^{(n)}$ is the subobject of $X$ consisting of those elements $\Omega[T]\rt X$ that factor through some tree $S$ of degree $\leq n$. For example, the boundary $\dau\Omega[T]$ of a representable presheaf is defined as the $(n-1)$-skeleton $\Omega[T]^{(n-1)}$, where $n$ is the degree of the tree $T$. Explicitly, $\dau\Omega[T](S)$ is the set of maps $S\rt T$ in $\Omega$ that factor through a non-invertible face map $S'\rt T$.

When $x\colon \Omega[T]\rt X$ is an element of $X$, define the boundary $\dau x$ of $x$ to be the restriction of $x$ to $\dau\Omega[T]$. The following is a straightforward variation of Lemma 2.6 in \cite{cis14}:
\begin{lemma}\label{lem:degeneraciesdeterminedbyboundary}
Let $x, y\colon \Omega[T]\rt X$ be two degenerate elements of a normal dendroidal set $X$. If the boundaries of $x$ and $y$ agree then $x$ and $y$ are the same.
\end{lemma}
\begin{proof}
Write $x=\sigma^*\ol{x}$ and $y=\tau^*\ol{y}$ where $\ol{x}\colon\Omega[S]\rt X$ and $\ol{y}\colon \Omega[S']\rt X$ are nondegenerate and $\sigma\colon T\rt S$ and $\tau\colon T\rt S'$ are non-invertible degeneracy maps. Let $\alpha$ be any section of $\sigma$ and let $\beta$ be any section of $\tau$. Since the boundaries of $x$ and $y$ agree, we have that
$$
\ol{x}= \alpha^*x = \alpha^*y = (\tau\alpha)^*\ol{y}
$$
and similarly $\ol{y}=(\sigma\beta)^*\ol{x}$. If the composite map $\tau\alpha$ in $\Omega$ could be factored as a non-invertible degeneracy map followed by a face map, then $\ol{x}$ would be a degenerate element. In other words, the map $\tau\alpha\colon S\rt T\rt S'$ has to be a face map and in particular the degree of $S$ is less than or equal to the degree of $S'$. Applying the same argument to the composite $\sigma\beta$ shows that the degrees of $S$ and $S'$ agree, which in turn implies that $\tau\alpha$ and $\sigma\beta$ are isomorphisms. Furthermore, we have that
$$
(\tau\alpha)^*(\sigma\beta)^*\ol{x} = (\tau\alpha)^*\ol{y} = \ol{x}
$$
Since $\ol{x}$ is a nondegenerate element of a normal dendroidal set, it has no nontrivial automorphisms. From this we conclude that
\setlength{\leftmargini}{2em}
\begin{itemize}
 \item[($\star$)] for \emph{any} choice of sections $\alpha\in \Gamma(\sigma)$ and $\beta\in \Gamma(\tau)$, the map $\sigma\beta$ is inverse to $\tau\alpha$.
\end{itemize}
We claim that $\sigma=\sigma\beta\tau$, in which case we conclude that
$$
y=\tau^*\ol{y} = \tau^*(\sigma\beta)^*\ol{x} = \sigma^*\ol{x} = x.
$$
To see that $\sigma=\sigma\beta\tau$, it suffices to check that both degeneracy maps have the same set of sections, by Lemma \ref{lem:degeneraciesdeterminedbysections}. Our conclusion ($\star$) shows that $\alpha$ is a section of $\sigma\beta\tau$ as soon as $\alpha$ is a section of $\sigma$. For the converse, suppose that $\alpha$ is a section of $\sigma\beta\tau$. Since $\sigma\beta$ is an isomorphism, we have that $\tau\alpha$ is an inverse to $\sigma\beta$ and consequently $\alpha\sigma\beta$ is a section of $\tau$. 

But now observe that ($\star$) implies that the isomorphism $\sigma\beta$ is actually independent of the chosen section $\beta\in \Gamma(\tau)$. This means that $\sigma\beta=\sigma(\alpha\sigma\beta)$, which in turn implies that $\alpha$ is a section of $\sigma$.
\end{proof}
When $X$ is a normal dendroidal set, the skeletal filtration \eqref{diag:skeletalfiltration} can be obtained by attaching cells \cite{ber08}. More precisely, each inclusion $X^{(n-1)}\rt X^{(n)}$ fits into a pushout square
$$\xymatrix{
\coprod_{|T_\alpha|=n} \dau\Omega[T_\alpha] \ar[r]\ar[d] & X^{(n-1)}\ar[d]\\
\coprod_{|T_\alpha|=n} \Omega[T_\alpha]\ar[r]^-{(x_\alpha)} & X^{(n)}.
}$$
We will call the resulting elements $x_\alpha\colon \Omega[T_\alpha]\rt X$ \emph{generating} nondegenerate elements. For every nondegenerate element $x$ of $X$ there is a unique generating nondegenerate element $x_\alpha$, together with a unique automorphism $\phi$ of $T_\alpha$ such that $x=\phi^*x_\alpha$.

\subsection*{Cylinders and homotopies}\label{sec:cylinders}
The category of dendroidal sets admits a left proper model structure (called the \emph{operadic} model structure) in which the cofibrations are the normal monomorphisms and the fibrant objects are the $\infty$-operads \cite{cis09}. For each dendroidal set $X$, let $J\otimes X$ be the Boardman-Vogt tensor product of $X$ with the dendroidal nerve $J$ of the groupoid $\{0\simeq 1\}$ with objects $0$ and $1$, together with a unique isomorphism between them. 

Taking the tensor product of $X$ with the functors $\{0, 1\}\rt J\rt *$ provides a factorization of the fold map
$$\xymatrix{
X\coprod X \simeq \{0, 1\}\otimes X\ar[r] & J\otimes X\ar[r]^-\sigma & X.
}$$
When $X$ is a normal dendroidal set, the first map is a cofibration and the second map is a weak equivalence. 

It follows immediately from this description that for any monomorphism $A\rt B$ between normal dendroidal sets, the map
$$\xymatrix{
J\otimes A\cup_{\{0, 1\}\otimes A} \{0, 1\}\otimes B\ar[r] & J\otimes B
}$$
is a cofibration. Similarly, the map
$$\xymatrix{
J\otimes A\cup_{\{i\}\otimes A} \{i\}\otimes B\ar[r] & J\otimes B
}$$
is a trivial cofibration for $i=0, 1$. These properties make sure that the notion of homotopy induced by the cylinder $J$ is well-behaved. For example, let $p\colon Y\rt X$ be a map and let $y_0, y_1\colon \Omega[T]\rt Y$ be two elements of $Y$. Then a fiberwise homotopy between $y_0$ and $y_1$, relative to the boundary $\dau\Omega[T]$, is given by a map $H$ which fits into a commuting diagram
$$\xymatrix{
J\otimes\dau\Omega[T] \ar[d]\ar[r]^\sigma & \dau\Omega[T] \ar[r] & Y\ar[d]^p\\
J\otimes\Omega[T]\ar[r]_\sigma\ar[rru]^H & \Omega[T] \ar[r] & X
}$$
whose restriction to $\{i\}\otimes \Omega[T]$ agrees with the map $y_i$ (for $i=0, 1$).
\begin{lemma}\label{lem:Jhomotopyisanequivalencerelation}
Let $p\colon Y\rt X$ be an operadic fibration. Then fiberwise homotopy relative to the boundary provides an equivalence relation on the set of elements $Y(T)$. More generally, for any monomorphism $A\rt B$ between normal dendroidal sets the notion of fiberwise homotopy relative to $A$ provides an equivalence relation on $\Hom(B, Y)$.
\end{lemma}
\begin{proof}
This is a standard argument using the homotopy extension and lifting property. For later reference (cf.~ the proof of Proposition \ref{prop:skeletalmodels}), we prove transitivity in a slightly more general setting. Let $x, y, z\colon B\rt Y$ be maps and suppose that there are fiberwise homotopies $g\colon x\simeq y$ and $h\colon y\simeq z$, where $h$ is a fiberwise homotopy rel $A$. Then $x\simeq z$ via a homotopy that agrees with $g$ when restricted to $A$. Indeed, consider the diagram
$$\xymatrix{
J\otimes \big(J\otimes A\cup \{0,1\} \otimes B\big) \cup \{0\}\otimes \big(J\otimes B\big)\ar[d]\ar[rr]^-{H} & & Y\ar[d]^p\\
J\otimes (J\otimes B)\ar[r]\ar@{..>}[rru]_-L & B\ar[r] & X
}$$
where the map $H$ is given by $H(s, t, a) = g(t, a)$ on $J\otimes(J\otimes A)$, while it is given on $J\otimes \big(\{0, 1\} \otimes B\big)\cup \{0\}\otimes \big(J\otimes B\big)$ by
$$
H(s, 0, b) = x(b) \qquad\qquad H(s, 1, b) = h(s, b) \qquad\qquad H(0, t, b) = g(t, b)
$$ 
(writing this without these formulas is a bit troublesome). Since $p$ is a fibration, there is a lift $L$ as indicated. The restriction of $L$ to $\{1\}\otimes (J\otimes B)$ provides a fiberwise homotopy between $x$ and $z$, which agrees with $g$ when restricted to $A$.
\end{proof}

\section{Existence of minimal fibrations}\label{sec:proof}
This section contains the proof of Theorem \ref{thm:existenceofminimalmodels}, which asserts that any fibration of dendroidal sets $Y\rt X$ admits a minimal fibration as a deformation retract, at least when $Y$ is normal. The idea of the proof is to construct a deformation retract of the fibration $Y\rt X$ which is \emph{skeletal} (Definition \ref{def:skeletalfibration}) by induction over the skeletal filtration of $Y$. We then show that any such skeletal fibration is a minimal fibration.

\subsection{Skeletal fibrations}
The following definition is an immediate analogue of the notion of `skeletality' appearing in the classical literature on simplicial sets (where it is usually called minimality, anticipating Corollary \ref{cor:minimalisskeletal}):
\begin{definition}\label{def:skeletalfibration}
Let $p\colon Y\rt X$ be an operadic fibration of dendroidal sets. We will say that $p$ is a \emph{skeletal} fibration if for any two elements $y_0, y_1\colon \Omega[T]\rt Y$ which are fiberwise homotopic relative to their boundary, there is an automorphism $\phi\in\mm{Aut}(T)$ such that $y_0=\phi^*y_1$.
\end{definition}
There is a second natural extension of the notion of `skeletality' to dendroidal sets, where one requires two homotopic elements to be equal. This condition is too restrictive for our purposes. Indeed, the following example demonstrates that there are dendroidal sets that cannot have a deformation retract satisfying this stricter condition of skeletality:
\begin{example}\label{ex:dendroidalsetswhicharenotstrictlyskeletal}
Let $C_2$ be the 2-corolla and let $\eta$ be tree with a single edge and no vertices. Their associated dendroidal sets are $\Omega[C_2]$ and $\Delta[0]:=\Omega[\eta]$. The 2-corolla $C_2$ has a single nontrivial automorphism $\tau$ of order 2 and its boundary $\dau \Omega[C_2]$ is the disjoint union of three edges. Define $J\otimes_\tau \Omega[C_2]$ to be the pushout
$$\xymatrix{
\Omega[C_2]\coprod \Omega[C_2]\ar@{>->}[r]\ar[d]_{(\mm{id}, \tau)} & J\otimes \Omega[C_2]\ar[d]^p\\
\Omega[C_2]\ar@{>->}[r] & J\otimes_\tau \Omega[C_2].
}$$
The bottom map defines an element of $J\otimes_\tau \Omega[C_2]$ which is $J$-homotopic to its conjugate by $\tau$. On the other hand, $J\otimes_\tau \Omega[C_2]$ is normal, since the top map in this pushout diagram is a normal monomorphism and $\Omega[C_2]$ is normal. Next, consider the pushout
$$\xymatrix{
J\otimes \dau \Omega[C_2]\ar[r]\ar[d] & J\otimes \Omega[C_2]\ar[r]^p & J\otimes_\tau \Omega[C_2]\ar[d]\\
\Delta[0] \ar[rr] & &  J\otimes_\tau \Omega[C_2]/J\otimes_\tau \dau\Omega[C_2]
}$$
Note that both $J\otimes \dau \Omega[C_2]$ and $\Delta[0]$ have no elements indexed by non-linear trees. Since pushouts of dendroidal sets are computed objectwise, this implies that the pushout $J\otimes_\tau \Omega[C_2]/J\otimes_\tau \dau \Omega[C_2]$ is again a normal dendroidal set.  Finally, let 
$$\xymatrix{
J\otimes_\tau \Omega[C_2]/J\otimes_\tau \dau \Omega[C_2]\ar@{>->}[r]^-\sim & X
}$$ 
be a fibrant-cofibrant replacement of this dendroidal set. The composite
$$\xymatrix{
x\colon \Omega[C_2]\ar@{>->}[r] & J\otimes_\tau \Omega[C_2] \ar[r] & J\otimes_\tau \Omega[C_2]/J\otimes_\tau \dau \Omega[C_2]\ar[r] & X
}$$
defines an element $x$ of $X$ with the property that $x$ is homotopic (relative to the boundary) to $\tau^*x$, while $x$ differs from $\tau^*x$ since $X$ was assumed normal. This property is shared by the image of $x$ under a retraction $r\colon X\rt M$. We conclude that any retraction of $X$ admits two distinct (but conjugate) 2-corollas which are homotopic relative to their boundary.
\end{example}
\begin{proposition}\label{prop:skeletalmodels}
Let $p\colon Y\rt X$ be a fibration of dendroidal sets with normal domain. Then $p$ admits a skeletal fibration $q\colon M\rt X$ as a fiberwise strong deformation retract (with respect to the functorial cylinder $J$).
\end{proposition}
\begin{proof}
We construct the inclusion $i\colon M\subseteq Y$, the retraction $r\colon Y\rt M$ and the strong deformation retraction $H\colon J\otimes Y\rt Y$ all at the same time, by induction along the skeleta of $Y$. Suppose that we have formed
$$\xymatrix{
M^{(n)} \ar@<1ex>@{^{(}->}[r]^{i^{(n)}} & Y^{(n)}\ar@<1ex>[l]^{r^{(n)}} & \text{and} & J\otimes Y^{(n)}\ar[r]^-{H^{(n)}} & Y
}$$
where $H_0^{(n)}$ is the inclusion of $Y^{(n)}$ into $Y$ and $H_1^{(n)}$ is the composite $i^{(n)}\circ r^{(n)}$. All maps are maps over the base $X$, where $M^{(n)}$ is considered as the domain of the map $p i^{(n)}\colon M^{(n)}\rt Y \rt X$. 

We start by producing $M^{(n+1)}$ and the inclusion $i^{(n+1)}\colon M^{(n+1)}\rt Y^{(n+1)}$. Recall that $Y^{(n+1)}$ is obtained from $Y^{(n)}$ by attaching a set of generating nondegenerate elements (together with their conjugates under the $\mm{Aut}(T)$-action). For each $T\in \Omega$ of degree $n+1$, let $N_T\subseteq Y(T)$ be the set of generating nondegenerate elements $y\colon \Omega[T]\rt Y$ such that
\begin{itemize}
 \item[(a)] the boundary $\dau y\colon \dau\Omega[T]\rt Y$ takes values in the subobject $M^{(n)}$.
 \item[(b)] $y$ is not fiberwise homotopic (relative to the boundary) to a degenerate element of $Y$.
\end{itemize}
We say that two elements $y_0$ and $y_1$ of $N_T$ are \emph{equivalent} if $y_0$ is fiberwise homotopic (relative the boundary) to $\phi^*y_1$, for some $\phi\in\text{Aut}(T)$. This defines an equivalence relation by Lemma \ref{lem:Jhomotopyisanequivalencerelation}.

We now construct $M^{(n+1)}$ and $i^{(n+1)}\colon M^{(n+1)} \rt Y^{(n+1)}$ by attaching one copy of $\Omega[T]$ to $M^{(n)}$ for every equivalence class of elements in the set $N_T\subseteq Y(T)$ and mapping it to a representative of that class in $Y(T)$. This defines $M^{(n+1)}$ together with an inclusion into $Y^{(n+1)}$. Furthermore, the resulting map out of the pushout
\begin{equation}\label{eq:inclusionofM}\xymatrix{
M^{(n+1)}\cup_{M^{(n)}} Y^{(n)} \ar[r] & Y^{(n+1)}
}\end{equation}
can be obtained as an iterated pushout of boundary inclusions $\dau\Omega[T]\rt \Omega[T]$. Indeed, we obtain this inclusion by attaching those generating nondegenerate elements of $Y$ which do not satisfy one of the above two conditions (a) or (b), as well as those generating nondegenerate elements in $N_T$ which are not yet contained in $M^{(n+1)}$.

Having constructed the inclusion $M^{(n+1)}\rt Y^{(n+1)}$, our next task is to extend the deformation retraction $H^{(n)}$. The constant homotopy on $M^{(n+1)}$ and the homotopy $H^{(n)}$ on $Y^{(n)}$ (relative to $M^{(n)}$) together define a homotopy
$$\xymatrix{
J\otimes \big(M^{(n+1)}\cup_{M^{(n)}} Y^{(n)}\big) \ar[r] & Y.
}$$
We extend this homotopy along each of the cell attachments that assemble into inclusion \eqref{eq:inclusionofM}.

\emph{Case 1:} we attach a generating nondegenerate element $y\colon \Omega[T]\rt Y$ which satisfies (a) and (b). By construction, this element is fiberwise homotopic (relative to the boundary) to an element which is contained in $M^{(n+1)}$. The extension of $H^{(n)}$ to the element $y$ is given by a choice of such a fiberwise homotopy (relative boundary).

\emph{Case 2:} we attach a generating nondegenerate element $y\colon \Omega[T]\rt Y$ which satisfies (a) but not (b). Then $y$ is fiberwise homopic (relative boundary) to a degenerate element, which is again contained in $M^{(n+1)}$. The extension of $H^{(n)}$ to the element $y$ is then given by a choice of such fiberwise homotopy (relative boundary).

\emph{Case 3:} we attach a generating nondegenerate element $y\colon \Omega[T]\rt Y$ which does not satisfy (a). Then we can form the commuting diagram
$$\xymatrix{
J\otimes \dau \Omega[T] \cup \{0\}\otimes \Omega[T]\ar[rr]^-{(H^{(n)}\Small\circ\dau y, y)}\ar[d] & & Y\ar[d]\\
J\otimes \Omega[T]\ar[r] & \Omega[T]\ar[r] & X
}$$
Since $Y\rt X$ is a fibration, there is a lift $h\colon J\otimes \Omega[T]\rt Y$ which provides a fiberwise homotopy between $y$ and an element $z\colon \Omega[T]\rt Y$ whose boundary is contained in $M^{(n)}$. By construction, this fiberwise homotopy extends the deformation retraction $H^{(n)}$ applied to the boundary of $y$. 

When $z$ is an element of $M^{(n+1)}$ we use the homotopy $h$ to extend $H^{(n)}$ to the element $y$. When $z$ is not contained in $M^{(n+1)}$, we have already constructed a homotopy $k$ (relative to the boundary) between $z$ and an element $m$ in $M^{(n+1)}$ in the previous two steps. We can compose the two homotopies $h$ and $k$ (as in the proof of Lemma \ref{lem:Jhomotopyisanequivalencerelation}) to produce a homotopy between $y$ and $m$ which agrees with $H^{(n)}$ on the boundary. Use this homotopy to extend $H^{(n)}$ over the element $y$.

In this way we produce a fiberwise strong deformation retraction
$$\xymatrix{
H^{(n+1)}\colon J\otimes Y^{(n+1)}\ar[r] & Y
}$$
extending $H^{(n)}$. By construction, the restriction of $H^{(n+1)}$ to $\{1\}\otimes Y^{(n+1)}$ factors as $i^{(n+1)}\circ r^{(n+1)}$ for some retraction $r^{(n+1)}\colon Y^{(n+1)}\rt M^{(n+1)}$. Proceeding by induction on the skeleton, we thus obtain a fiberwise strong deformation retract of $Y$ onto some subobject $M$.

It remains to check that the resulting map $q=pi\colon M\rt X$ is a skeletal fibration. Since $M$ is a fiberwise retract of $Y$, it follows that the map $q\colon M\rt X$ is a fibration. To see that it is skeletal, let $x, y\colon \Omega[T]\rt M$ be two elements of $M$ that are fiberwise homotopic relative to their boundary. If both maps are degenerate, then they must be the same since their boundaries are the same (Lemma \ref{lem:degeneraciesdeterminedbyboundary}). We may therefore assume that $x$ is nondegenerate.

Applying the inclusion $i$, we see that $x$ and $y$ determine homotopic elements (relative boundary) in $Y$, whose boundary lies in the $n$-skeleton $M^{(n)}$. By the construction of $M^{(n+1)}$ it follows that $y$ is nondegenerate as well: indeed, we removed all nondegenerate elements from $Y$ that were fiberwise homotopic (relative boundary) to degenerate elements in $M$. But then the construction of $M$ implies that $x=\phi^*y$ for some $\phi\in\mm{Aut}(T)$, since we attached only one generating nondegenerate element to $M^{(n)}$ for each equivalence class of nondegenerate elements in $Y$.
\end{proof}

\subsection{Minimal fibrations}
Let $p\colon Y\rt X$ be a skeletal fibration with normal domain. To check that $p$ is a minimal fibration, it suffices to check that any fiberwise self-homotopy equivalence of $p$ is an isomorphism. In turn, this is guaranteed by the following
\begin{proposition}\label{prop:skeletalfibrationsareminimal}
Let $p\colon Y\rt X$ be a skeletal fibration with normal domain. If $f\colon Y\rt Y$ is an endomorphism of $p$ which is fiberwise homotopic to the identity map on $Y$, then $f$ is an isomorphism.
\end{proposition}
We prove this by induction on the skeleton of $Y$, the case $Y^{(-1)}$ being trivial. The inductive step follows from the following two lemmas:
\begin{lemma}
Let $p\colon Y\rt X$ be a skeletal fibration with normal domain $Y$ and let $h\colon J\otimes Y\rt Y$ be a fiberwise homotopy from an endomorphism $h_0$ of $p$ to the identity map. If $h_0$ induces an isomorphism on the $n$-skeleton $Y^{(n)}$, then $h_0$ is injective on the $(n+1)$-skeleton of $Y$.
\end{lemma}
\begin{proof}
Let $T\in\Omega$ be of degree $n+1$ and let $x, y\colon \Omega[T]\rt Y$ be two elements such that $h_0x=h_0y$. We have that $px=py$ and by inductive hypothesis $\dau x=\dau y$. We may clearly assume that one of the two, say $x$, is a nondegenerate element.

Take the (fiberwise) homotopies $h(x)$ and $h(y)$ from $h_0x$ to $x$ and from $h_0y$ to $y$, together with the constant homotopies on $\dau x=\dau y$ and $h_0(x)=h_0(y)$. Together these give a map $K$ which fits into a commuting square
$$\xymatrix{
J\otimes \big(J\otimes\dau \Omega[T]\cup \{0, 1\}\otimes \Omega[T]\big) \cup \{0\}\otimes \big(J\otimes \Omega[T]\big)\ar[d]\ar[rr]^-K & & Y\ar[d]\\
J\otimes \big(J\otimes \Omega[T]\big)\ar[r] \ar@{..>}[rru]_-L & \Omega[T]\ar[r] & X.
}$$
This square allows for a diagonal map $L$ because $Y\rt X$ is a fibration. The composite
$$\xymatrix{
\{1\}\otimes \big(J\otimes \Omega[T]\big) \ar[r] & J\otimes(J\otimes \Omega[T])\ar[r]^-L & Y
}$$
determines a fiberwise homotopy between $x$ and $y$, relative to $\dau \Omega[T]$. The fibration $p$ is skeletal, so $x=\sigma^*y$ for some $\sigma\in\mm{Aut}(T)$. But then we also have that $h_0(x)=\sigma^*h_0(y)= \sigma^*h_0(x)$. Since $Y$ is normal, this either means that $\sigma=1$ (if $h_0(x)$ is nondegenerate) or $h_0(x)$ is degenerate. 

In the latter case, there is a \emph{degenerate} $z$ such that $h_0(x)=h_0(z)$, since $h_0$ was assumed to be an isomorphism on $Y^{(n)}$. Repeating the previous argument shows that $x=\tau^*z$ for some $\tau\in \mm{Aut}(T)$. Since $x$ was assumed nondegenerate, this cannot happen and we conclude again that $\sigma=1$. This shows that $x=y$.
\end{proof}
\begin{lemma}
Let $f\colon Y\rt Y$ be a fiberwise homotopy equivalence from a skeletal fibration $p\colon Y\rt X$ with normal domain to itself. If $f$ induces an isomorphism on $Y^{(n)}$, then $f$ induces a surjective map on elements of degree $n+1$.
\end{lemma}
\begin{proof}
Let $f\colon Y\rt Y$ be a fiberwise homotopy equivalence from $p$ to itself. Factor $f=qi$ where $i\colon Y\rt Z$ is a cofibration and $q\colon Z\rt Y$ is a trivial fibration. Since $Y$ is normal, so is $Z$ and $i$ is the inclusion of a fiberwise strong deformation retract over $X$, with retraction $r\colon Z\rt Y$ over $X$.

Let $T\in\Omega$ be of degree $n+1$ and take $x\colon \Omega[T]\rt Y$. Because $f$ induces an isomorphism on the $n$-skeleton of $Y$, there is a map $y\colon \dau \Omega[T]\rt Y$ such that $fy= \dau x$. Since $q$ is a trivial fibration, there is a map $z\colon \Omega[T]\rt Z$ such that $qz=x$ and $\dau z = iy$:
$$\xymatrix{
\dau \Omega[T]\ar[d]\ar[r]^y & Y\ar[r]^i & Z\ar[d]^q\\
\Omega[T]\ar[rr]_x\ar@{..>}[rru]|-z & & Y.
}$$
Let $w=r(z)$. Then $\dau f(w) = fr(\dau z) = fri(y) = \dau x$, and $f(w)$ is fiberwise homotopic to $x$ (rel $\dau \Omega[T]$):
$$\xymatrix{
f(w) = fr(z) = qir(z) \simeq_{\text{rel } \dau \Omega[T]} q(z) = x
}$$
Since $p$ was skeletal it follows that $f(w)=\sigma^*x$ for some $\sigma\in\mm{Aut}(T)$, so that $x=f(\sigma^{-1 *}w)$.
\end{proof}
\begin{proof}[Proof (of Theorem \ref{thm:existenceofminimalmodels})]
Let $p\colon Y\rt X$ be an operadic fibration with normal domain. By Proposition \ref{prop:skeletalmodels} $p$ admits a skeletal fibration $q\colon M\rt X$ as a fiberwise strong deformation retract, with inclusion $i\colon M\rt Y$ and retraction $r\colon Y\rt M$. The object $M$ is normal, being the retract of a normal object. It then follows from Proposition \ref{prop:skeletalfibrationsareminimal} that $q$ is a minimal fibration.

It remains to check that the retraction $r\colon Y\rt M$ is a trivial fibration when the base $X$ is a normal dendroidal set. This is proven exactly as in Quillen's paper \cite{qui68}, which treats the analogous result for simplicial sets. Consider a diagram of the form
\begin{equation}\label{diag:trivfib}\vcenter{\xymatrix{
\dau \Omega[T]\ar[d]\ar[r]^{y} & Y\ar[d]^r\\
\Omega[T]\ar[r]_x & M
}}\end{equation}
Then $ix$ provides a lift making the bottom triangle commute, but the boundary of $ix$ agrees with $ir y\colon \dau\Omega[T]\rt Y$, which is only fiberwise homotopic to $y$ using the the deformation retraction $H$ between $ir$ and the identity on $Y$.

We therefore replace $ix$ by a homotopic element of $Y$ whose boundary agrees with $y$. Since $p\colon Y\rt X$ is a fibration, there is a lift in the diagram
$$\xymatrix{
\{0\}\otimes \Omega[T]\cup J\otimes\dau \Omega[T]\ar[r]^-{(ix, H)}\ar[d] & Y\ar[d]\\
J\otimes \Omega[T]\ar[r]\ar@{..>}[ru]_K & X.
}$$
Let $z\colon \Omega[T]\rt Y$ be the restriction of the lift $K$ to $\{1\}\otimes \Omega[T]$. We claim that $z\colon \Omega[T]\rt Y$ provides a lift in diagram \eqref{diag:trivfib}. 

Indeed, $\dau z = y$ and the deformation retraction $H$ gives a homotopy from $ir(z)$ to $z$. This means that $i(x)$ and $ir(z)$ are fiberwise homotopic to $z$, both via a homotopy which is given by $H$ when restricted to the boundary $\dau\Omega[T]$. But then there is a fiberwise homotopy between $ir(z)$ and $ix$ which is constant on the boundary (using an argument similar to the proof of Lemma \ref{lem:Jhomotopyisanequivalencerelation}). It follows that $r(z)$ is homotopic (relative boundary) to $x$. Because $q\colon M\rt X$ was a skeletal fibration, we conclude that $x=\phi^*r(z)$ for some automorphism $\phi$ of $T$. 

Applying $q$, we see that $q(x) = \phi^*qr(z) = \phi^*q(x)$. But $X$ is a normal dendroidal set, so $\phi$ must be the identity automorphism. We conclude that $x=r(z)$, which means that $z\colon \Omega[T]\rt Y$ provides a diagonal lift in diagram \eqref{diag:trivfib}.
\end{proof}
\begin{corollary}\label{cor:minimalisskeletal}
Let $p\colon Y\rt X$ be a fibration with normal domain. Then $p$ is a minimal fibration iff $p$ is a skeletal fibration.
\end{corollary}
\begin{proof}
All skeletal fibrations are minimal fibrations, so assume that $p$ is a minimal fibration. Then there is a trivial cofibration $i\colon M\rt Y$ such that $pi$ is a minimal fibration. By minimality of $p$, the map $i$ is an isomorphism and one finds that $p$ is skeletal.
\end{proof}
\begin{remark}\label{rem:varyingcylinder}
In particular, the notion of skeletal fibration from Definition \ref{def:skeletalfibration} is independent of the chosen cylinder, as long as it preserves colimits and has the properties mentioned in Section \ref{sec:cylinders}.
\end{remark}
\begin{corollary}\label{cor:pullbacksofminimalfibrations}
Let $f\colon X'\rt X$ be a map of dendroidal sets with the property that for any element $x\colon \Omega[T]\rt X'$, the map $\mm{Aut}(x)\rt \mm{Aut}(fx)$ is bijective. If $Y\rt X$ is a minimal fibration with normal domain, then the base change $f^*Y\rt X'$ is a minimal fibration as well.
\end{corollary}
\begin{proof}
This follows immediately from the corresponding property for skeletal fibrations: indeed, let $p\colon Y\rt X$ be a skeletal fibration (with normal domain) and consider the pullback square
$$\xymatrix{
f^*Y\ar[r]^{f'}\ar[d]_{p'} & Y\ar[d]^p\\
X'\ar[r]_f & X.
}$$
If $x, y\colon \Omega[T]\rt f^*Y$ are fiberwise homotopic, then $f'x$ and $f'y$ are fiberwise homotopic as well. It follows that there is an element $\phi\in\mm{Aut}(T)$ such that $f'x=\phi^*f'y$. Projecting to $X$, we find that $\phi$ is an automorphism of the element $pf'x=fp'x\colon \Omega[T]\rt X$. By the assumption that $f$ induces a bijection on automorphism groups, it follows that $\phi^*p'y=\phi^*p'x=p'x$. Since $Y'$ is the pullback of $Y$ and $X'$ over $X$, this implies that $x=\phi^*y$. We conclude that $Y'\rt X'$ is indeed skeletal.
\end{proof}
\begin{example}
The condition of Corollary \ref{cor:pullbacksofminimalfibrations} is satisfied by monomorphisms and by all maps between normal dendroidal sets (whose elements all have trivial automorphism groups). Furthermore, it is satisfied by all maps whose domain is a simplicial set, i.~e.~ a dendroidal set without elements indexed by nonlinear trees. In particular, if $p\colon X\rt S$ is a minimal fibration, then the fiber $X_c$ of $p$ over a colour $c\colon \Delta[0]\rt S$ is a minimal $\infty$-category.
\end{example}
\begin{remark}\label{rem:minimalfibrationsarenotpullbackstable}
Minimal (or skeletal) fibrations are not stable under base change along an arbitrary map. For example, consider the normal $\infty$-operad $X$ constructed in example \ref{ex:dendroidalsetswhicharenotstrictlyskeletal} and let $M$ be a skeletal deformation retract of it. The dendroidal set $M$ comes equipped with a 2-corolla $x\colon \Omega[C_2]\rt M$ which is homotopic (relative to its boundary) to $\tau^*x$, where $\tau$ is the nontrivial automorphism of $C_2$. 

Now let $p\colon E_\infty\rt {*}$ be a trivial fibration with normal domain. Then the map $M\rt *$ is a skeletal fibration, but the base change $M\times E_\infty\rt E_\infty$ is not. Indeed, let $y\colon \Omega[C_2]\rt E_\infty$ be a lift of the unique map $\Omega[C_2]\rt *$. Then the element $(x, y)\colon \Omega[C_2]\rt M\times E_\infty$ is fiberwise homotopic (rel.~ boundary) to $(\tau^*x, y)$, but it is not related to $(\tau^*x, y)$ via an automorphism of $C_2$.
\end{remark}

\section{Applications}\label{sec:applications}
By way of example, we give two applications to the theory of left fibrations between dendroidal sets.

\subsection{Gluing left fibrations}
Let $X$ be a simplicial set and let $A$ and $B$ be two subobjects of $X$ which cover $X$. The class of Kan fibrations satisfies a certain `homotopy descent' condition, which asserts that Kan fibrations over $A$ and $B$ can be glued - up to homotopy - to yield a fibration over their union $X$. More precisely, consider two Kan fibrations $Y_A\rt A$ and $Y_B\rt B$ and a homotopy equivalence between their restrictions to the intersection $A\cap B$. Then there exists a Kan fibration $Y\rt X$ whose restrictions to $A$ and $B$ are homotopy equivalent to the original two fibrations. 

This homotopy descent property reflects the fact that Kan fibrations are local in nature: a map $Y\rt X$ is a Kan fibration whenever its restriction to each simplex of $X$ is a Kan fibration. One may therefore expect a similar gluing result to hold for fibrations between dendroidal sets which have the same locality property. Operadic fibrations do not have this property, but left fibrations do since they are defined by the right lifting property with respect to subobjects of representables \cite{heu12}.

With this in mind, the homotopy descent property for left fibrations of dendroidal sets follows by a straightforward reduction to the situation where all left fibrations are minimal.
\begin{proposition}\label{prop:gluingleftfibrations}
Consider a diagram of dendroidal sets
\begin{equation}\label{diag:pushoutdiagramofleftfibrations}\vcenter{\xymatrix@R=1.8pc{
Y_1\ar@{->>}[d] & Y_0\ar[l]\ar[r]\ar@{->>}[d] & Y_2\ar@{->>}[d]\\
X_1 & X_0\ar@{>->}[l]\ar@{>->}[r] & X_2
}}\end{equation}
in which the vertical maps are left fibrations and the bottom horizontal maps are cofibrations. Suppose that both squares are `homotopy cartesian', in the sense that the maps $Y_0\rt Y_i\times_{X_i} X_0$ are weak equivalences in the covariant model structure over $X_0$. Then there exists a left fibration over the pushout $X_1\cup_{X_0} X_2$, whose pullback to each of the $X_i$ is weakly equivalent to the left fibration $Y_i\rt X_i$ in the covariant model structure over $X_i$.
\end{proposition}
\begin{proof}
We can replace the above diagram by any weakly equivalent diagram of left fibrations over the $X_i$. In particular, we can assume that all dendroidal sets $Y_i$ are cofibrant.

We can further reduce to the case where all vertical maps are minimal left fibrations. Indeed, we can first replace Diagram \eqref{diag:pushoutdiagramofleftfibrations} by a diagram of the form
$$\xymatrix@R=1.9pc{
& M_0\ar@{..>}[d]^j\ar[ld]\ar[rd] & \\
Y_1\ar@{->>}[d] & Y_0\ar@{..>}[l]\ar@{..>}[r]\ar@{->>}[d] & Y_2\ar@{->>}[d]\\
X_1 & X_0\ar[l]\ar[r] & X_2
}$$
where $j$ is the inclusion of a minimal fibration with cofibrant domain. The resulting diagram of left fibrations remains homotopy cartesian. Next, replace this diagram by a diagram of the form
\begin{equation}\label{diag:finaldiagramofminimalleftfibrations}\vcenter{\xymatrix@R=1.9pc{
Y_1\ar@{..>}[d]_{r_1} & M_0\ar@{..>}[l]\ar@{..>}[r] \ar[ld]\ar[rd]\ar@{->>}[dd] & Y_2\ar@{..>}[d]^{r_2}\\
M_1\ar@{->>}[d] & & M_2\ar@{->>}[d]\\
X_1 & X_0\ar[l]\ar[r] & X_2
}}\end{equation}
where $r_1$ and $r_2$ are fiberwise retractions onto minimal fibrations (with cofibrant domains). The vertical maps in the resulting diagram remain left fibrations and the maps $M_0\rt M_i\times_{X_i} X_0$ are given by the composition
$$\xymatrix@C=1.5pc{
M_0\ar[r] & Y_0\ar[r] & Y_i\times_{X_i} X_0 \ar[r] & M_i\times_{X_i} X_0.
}$$
The composition of the first two maps is an operadic weak equivalence and the second map is the base change of a fiberwise deformation retract over $X_i$. It follows that the composite is a weak equivalence between two minimal fibrations over $X_0$, which means that it must be an isomorphism.  In other words, the two solid squares in Diagram \eqref{diag:finaldiagramofminimalleftfibrations} are both pullback squares.

Taking the pushout of the top and bottom row gives a map
$$\xymatrix{
p\colon M_1\cup_{M_0} M_2\ar[r] & X_1\cup_{X_0} X_2.
}$$
Because both squares in Diagram \eqref{diag:finaldiagramofminimalleftfibrations} are cartesian, the pullback of this map to each of the $X_i$ reproduces the fibration $M_i\rt X_i$, up to a canonical isomorphism. Since left fibrations between dendroidal sets are local, it follows that the map $p$ is a left fibration over $X_1\cup_{X_0} X_2$ whose pullback to each of the $X_i$ is weakly equivalent to the original left fibration $Y_i\rt X_i$.
\end{proof}
Proposition \ref{prop:gluingleftfibrations} has a simple model-categorical consequence, which we will now explain. The covariant model structures over all dendroidal sets assemble into a functor
$$\entrymodifiers={+!!<0pt,\fontdimen22\textfont2>}\xymatrix{
\text{Alg}\colon \dSet^{\text{op}}\ar[r] & \cat{ModelCat}^{\text{R}}; \hspace{4pt} X \ar@{|->}[r] & \big(\cat{dSet}/X\big)^{\text{cov}}
}$$
Given a cospan of dendroidal sets
$$\xymatrix{
X_1 & X_0\ar[l]_f\ar[r]^{g} & X_2
}$$
we thus obtain a span of (combinatorial, left proper) model categories and right Quillen functors between them
$$\xymatrix{
\mm{Alg}(X_1)\ar[r]^-{f^*} & \mm{Alg}(X_0) & \mm{Alg}(X_2)\ar[l]_-{g^*}
}$$
Any such diagram of right Quillen functors admits a `homotopy pullback' model category $\mm{Alg}(X_1)\times^h_{\mm{Alg}(X_0)} \mm{Alg}(X_2)$, whose underlying category is the lax pullback of the above diagram of categories \cite{bar12}. More precisely, the homotopy limit model category has objects given by triples of objects $Y_i\in \mm{Alg}(X_i)$ together with two structure maps in $\mm{Alg}(X_0)$
$$\xymatrix{
\alpha\colon Y_0\ar[r] &  f^*Y_1 & \beta\colon Y_0\ar[r] & g^*Y_2.
}$$
The maps are maps of triples $Y_i\rt Z_i$ that are compatible with the two structure maps. This category carries a model structure in which the trivial fibrations are triples of trivial fibrations $Y_i\rt Z_i$, while the fibrant objects are given by triples of fibrant objects $Y_i$, together with structure maps $\alpha$ and $\beta$ which are \emph{weak equivalences}.

In the present situation, where each of the categories $\mm{Alg}(X_i)$ is just the category of dendroidal sets over $X_i$, this means that the category underlying the homotopy pullback $\mm{Alg}(X_1)\times^h_{\mm{Alg}(X_0)} \mm{Alg}(X_2)$ is simply the overcategory
$$
\left(\cat{dSet}^{1\leftarrow 0\rightarrow 2}\right)/X
$$
whose objects are diagrams of shape \eqref{diag:pushoutdiagramofleftfibrations}. The model structure described above agrees with the model structure for which
\begin{itemize}
 \item cofibrations are projective cofibrations between the underlying diagrams of dendroidal sets.
 \item fibrant objects natural transformations $Y\rt X$ such that each $Y_i\rt X_i$ is a left fibration and each map $Y_0\rt Y_i\times_{X_i} X_0$ is a covariant weak equivalence over $X_0$.
 \item weak equivalences between fibrant objects are degreewise weak equivalences.
\end{itemize}
This homotopy pullback model category comes equipped with a Quillen pair
$$\xymatrix{
\colim\colon \left(\dSet^{1\leftarrow 0 \rt 2}\right)/X \ar@<1ex>[r] & \dSet/\colim X \ar@<1ex>[l]
}$$
to the covariant model structure over the pushout of the diagram $X$. The right adjoint sends a map over $\colim X$ to its pullbacks to each of the $X_i$. Proposition \ref{prop:gluingleftfibrations} now has the following reformulation:
\begin{corollary}
When $X_1\leftarrow X_0\rt X_2$ is a diagram of cofibrations between dendroidal sets, the above Quillen pair is a Quillen equivalence.
\end{corollary}
\begin{proof}
The derived unit is easily checked to be a natural weak equivalence and the proof of Proposition \ref{prop:gluingleftfibrations} shows that the derived counit map is a natural weak equivalence. 
\end{proof}
\begin{remark}
The same result holds when only one of the two arrows is a cofibration. Indeed, this follows from the fact that the operadic model structure is left proper \cite{cis09}, the covariant model structures over weakly equivalent dendroidal sets are Quillen equivalent \cite{heu12} and the fact that two (naturally) Quillen equivalent diagrams of model categories have Quillen equivalent homotopy pullbacks \cite{bar12}.
\end{remark}

\subsection{Weak equivalences between left fibrations}
As another application, we give an alternative, self-contained proof of the result from \cite{heu12} that the weak equivalences between left fibrations of dendroidal sets are precisely the fiberwise weak equivalences.
\begin{proposition}\label{prop:leftfibrationwithcontractiblefibers}
Let $p\colon X\rt S$ be a left fibration between normal dendroidal sets. Then $p$ is a trivial fibration iff for every colour $c\colon \eta\rt S$, the fiber $X_c = X\times_S \eta$ is a contractible Kan complex.
\end{proposition}
\begin{corollary}
Consider a map of left fibrations over $S$
$$\xymatrix{
X\ar[rr]^f\ar[rd] & & Y\ar[ld]^q\\
& S &
}$$
Then $f$ is a weak equivalence iff for every colour $c\colon \eta\rt S$ the map between fibers $X_c\rt Y_c$ is a weak equivalence of Kan complexes.
\end{corollary}
\begin{proof}
By Brown's lemma, weak equivalences between left fibrations are preserved by the right Quillen functor taking the base change along $c\colon \eta\rt S$. A weak equivalence between left fibrations is therefore a fiberwise weak equivalence. For the converse, we can assume that $X$ and $Y$ are normal. Factor the fiberwise weak equivalence $f$ as a covariant trivial cofibration $X\rt \tilde{X}$, followed by a covariant fibration $\tilde{X}\rt Y$. Because $\tilde{X}\rt S$ is a left fibration, the trivial cofibration $X\rt \tilde{X}$ is a fiberwise weak equivalence. This implies that the covariant fibration $\tilde{X}\rt Y$ is a fiberwise weak equivalence as well.

For every colour $c\colon \eta\rt Y$, the fiber $\tilde{X}\times_Y \{c\}$ is isomorphic to the fiber over $\{c\}$ of the map between simplicial sets
$$\xymatrix{
\tilde{X}\times_S \{qc\} \ar[r] & Y\times_S \{qc\}.
}$$
This map is a trivial fibration between Kan complexes because $\tilde{X}\rt Y$ is a left fibration and a fiberwise weak equivalence. The left fibration $\tilde{X}\rt Y$ has contractible fibers, so it follows from Proposition \ref{prop:leftfibrationwithcontractiblefibers} that it is a trivial fibration. 
\end{proof}
\begin{proof}[Proof (of Proposition \ref{prop:leftfibrationwithcontractiblefibers})]
Let $i\colon M\rt X$ be the inclusion of a minimal fibration into $X$. The fibers of $pi\colon M\rt S$ are weakly equivalent to those of $p\colon X\rt S$ and the composite $pi$ is a trivial fibration iff the original fibration $p$ is a trivial fibration. We can therefore reduce to the case where $p\colon X\rt S$ is a minimal fibration with normal codomain. Note that a minimal fibration with contractible fibers actually has trivial fibers, i.~e.~ its fibers are isomorphic to $\eta$. 

We will prove that any minimal left fibration $p\colon X\rt S$ with trivial fibers is an isomorphism. It is immediate that $p$ induces a bijection on colours and a left fibration inducing a surjection on colours is always an epimorphism. We show by induction that $p$ induces a monomorphism on all $n$-skeleta. 

Assume that $p\colon X\rt S$ induces an isomorphism on $(n-1)$-skeleta and let $\alpha\colon \Omega[T]\rt S$ be a (possibly degenerate) element of degree $n$. We have to show that $\alpha$ has a unique lift to $X$ (in particular, this unique lift is degenerate if $\alpha$ was). The proof of this uses another inductive argument: we will say that an element $\alpha\colon \Omega[T]\rt S$ \emph{has a trunk of height} $0\leq k\leq n+1$ if there is a factorization
$$\xymatrix{
\Omega[T]\ar[r]^-\alpha\ar@{=}[d] & S\\
\Omega[T']\star \Delta[k-1] \ar[r] & \Omega[T']\ar[u]^{\ol{\alpha}}
}$$
The left vertical isomorphism asserts that there is a tree $T'$ of degree $n-k$ such that $T$ is obtained from $T'$ by adding a vertex below the root and grafting the result on top of a linear order (of degree $k-1$). When $k=n+1$, this means that $\Omega[T]=\Delta[n]$ is itself a linear order and when $k=0$ there is no condition. The bottom horizontal map is the degeneracy map obtained by removing all vertices below the root edge of $T'$.

If $\alpha$ has a trunk of height $n+1$, then $\alpha\colon \Delta[n]\rt S$ is a degenerate $n$-simplex in $S$. Such a simplex indeed has a unique (fully degenerate) lift since the fibers of $X\rt S$ over each colour are trivial. We proceed by \emph{decreasing} induction on the height of the trunk of $\alpha$. 

Suppose that $\alpha$ has a trunk of height $k$ and suppose that $\beta, \beta'\colon \Omega[T]\rt X$ are two lifts of $\alpha$ to $X$. Since $p\colon X\rt S$ induces an isomorphism between $(n-1)$-skeleta, the boundaries of $\beta$ and $\beta'$ are the same. 

Pick any leaf vertex $v$ of the tree $T$. We will construct a fiberwise homotopy between the elements $\beta$ and $\beta'$, which is constant on all the faces of $T$ except the face opposite $v$. To this end, consider the pushout square
$$\xymatrix{
\Lambda^v[T]\coprod \Lambda^v[T]\ar[r]^-{(\dau_0, \dau_1)} \ar[d] & \Lambda^v[T]\star \Delta[0]\ar[d] \\
\Omega[T]\coprod \Omega[T] \ar[r] & \Lambda^v[T\star \eta]
}$$
where $T\star \eta$ is the obtained from the tree $T$ by adding an extra root vertex $r$. The map $\dau_0$ (resp.~ $\dau_1$) is the face map induced by contracting the edge `$0$' above the extra root vertex (resp.~ by removing the extra root vertex and the root edge `$1$'). Denoting by $\sigma_r$ the degeneracy associated to the root vertex of $T\star \eta$, there are two maps
$$\xymatrix{
\Omega[T]\coprod \Omega[T]\ar[r]^-{(\beta, \beta')} & X & \Lambda^v[T]\star \Delta[0]\ar[r]^-{\sigma_r} & \Lambda^v[T]\ar[r]^-{\Lambda^v(\beta)} & X
}$$
which agree on $\Lambda^v[T]\coprod \Lambda^v[T]$. The associated map out of the pushout fits onto a commuting diagram
$$\xymatrix{
\Lambda^v[T\star \eta]\ar[d]\ar[rr] & & X\ar[d]\\
\Omega[T\star \eta]\ar[r]\ar@{..>}[rru]^H & \Omega[T]\ar[r]_\alpha & S.
}$$
Since $p$ is a left fibration, there exists a lift $H$ as indicated. Using the terminology from \cite{moe09}, $H$ gives a fiberwise homotopy along the 0-edge between $\beta$ and $\beta'$.

The restriction of $H$ to the remaining face $\dau^v\Omega[T]\star \Delta[0]$ gives a homotopy (rel boundary) from the face $\dau^v\beta$ to the face $\dau^v\beta'$. This restriction fits into a commutative diagram
$$\xymatrix{
\dau^v\Omega[T]\star \Delta[0]\ar[rr]^{\dau^v H}\ar@{=}[d] & & X\ar[d]\\
\dau^v\Omega[T']\star \Delta[k-1]\star \Delta[0] \ar[r] & \dau^v\Omega[T']\ar[r]_-{\ol{\alpha}} & S.
}$$
The left equality uses that the element $\alpha$ had a trunk of height $k$. In other words, the element $\dau^v H$ provides a lift of the degenerate element 
$$\xymatrix{
\dau^v\Omega[T']\star\Delta[k]=\dau^v\Omega[T']\star \Delta[k-1]\star \Delta[0]\ar[r] & \dau^v\Omega[T']\ar[r]^-{\ol{\alpha}} & S
}$$
which is of degree $n$ and has a trunk of height $k+1$. But by the inductive assumption, elements with a trunk of height $k+1$ have unique lifts. It follows that $H$ is degenerate, which means that the homotopy $H$ is also constant on the remaining face $\dau^v\Omega[T]$.

The map $H\colon \Omega[T\star \eta]\rt X$ thus provides a fiberwise homotopy (along the 0-edge) between $\beta$ and $\beta'$ which is contant on the boundary. This is not quite a homotopy in the sense of Section \ref{sec:cylinders}, but one can deduce the existence of such a homotopy either from Theorem B.2 in \cite{cis09}, or use the following argument. Observe that $H$ and the constant homotopy
$$\xymatrix{
\Omega[T\star \eta] \ar[r]^-{\sigma_r} & \Omega[T] \ar[r]^-{\beta'} & Y 
}$$
have the same boundary, except for the face obtained by contracting the edge `$0$' above the root vertex (on which the value of $H$ was $\beta$, rather than $\beta'$). Both of these homotopies therefore provide a diagonal lift for the same diagram
$$\xymatrix{
\Lambda^0[T\star \eta]\ar[r]\ar[d] & Y\ar[d]\\
\Omega[T\star \eta] \ar[r] & X
}$$
where $\Lambda^0[T\star \eta]$ excludes the face that contracting the `$0$'-edge. But lifts along inner horn inclusions are unique up to fiberwise $J$-homotopy, relative to the horn. In particular, we have that $\dau_0H= \beta$ and $\dau_0(\sigma_r\beta')=\beta'$ are fiberwise $J$-homotopic relative to their boundary.

Knowing that $\beta$ and $\beta'$ are $J$-homotopic relative to their boundary, we can now use that $p\colon X\rt S$ is a minimal fibration. It follows that $\beta=\phi^*\beta'$ for some automorphism $\phi$ of the tree $T$. Using $p$ to project to $S$, we find that $\phi$ induces an automorphism of the element $\alpha$. But $S$ was assumed to be normal, so $\phi$ is the identity and $\alpha$ indeed has a unique lift.
\end{proof}

\appendix

\section{Other examples}
\setcounter{theorem}{0}

Recently Cisinski \cite{cis14} has shown that the theory of minimal fibrations of simplicial sets can be generalised to model categories of presheaves over certain `Eilenberg-Zilber type' Reedy categories (\cite{cis14}, 2.1), in which the cofibrations are the monomorphisms. Such Reedy categories share the combinatorial properties of the simplex category $\Delta$ that provide presheaves over them with a well-behaved skeletal filtration, the crucial tool used in the construction of minimal Kan fibrations.

The model structure on dendroidal sets does not entirely fit into this framework for the simple reason that the category $\Omega$ of trees is not a strict Reedy category. Our proof of Theorem \ref{thm:existenceofminimalmodels} demonstrates how to take care of the automorphisms in $\Omega$ during the construction of minimal fibrations. Our treatment of automorphisms in Section \ref{sec:proof} extends to the following kind of `Eilenberg-Zilber type' \emph{generalised} Reedy categories, of which $\Omega$ is an example by Lemma \ref{lem:degeneraciesdeterminedbysections}.
\begin{definition}
A generalised Reedy category $\cat{R}$ is called an Eilenberg-Zilber category if it satisfies the following two conditions:
\begin{enumerate}
 \item $\cat{R}^-$ is the subcategory of split epimorphisms.
 \item two maps $r\rt s$ in $\cat{R}^-$ are the same if they have the same set of sections.
\end{enumerate}
\end{definition}
\begin{example}
Apart from the category of trees $\Omega$, the class of Eilenberg-Zilber categories includes well-known examples like the simplex category $\Delta$, Segal's category $\Gamma$ \cite{seg74}, Connes' cyclic category $\Lambda$ \cite{con83}, the category of nonempty finite sets and all group(oid)s.

Furthermore, the product of two Eilenberg-Zilber categories is also one and for any presheaf $X$ on an Eilenberg-Zilber category $\cat{R}$, the category of elements $\cat{R}/X$ is again one.
\end{example}
Indeed, all the definitions in Section \ref{sec:preliminariesondendroidalsets} make sense when $\Omega$ is replaced by an arbitrary Eilenberg-Zilber category, as defined above. In particular, any normal presheaf $X$ on an Eilenberg-Zilber category $\cat{R}$ admits a skeletal filtration, in which each inclusion $X^{(n)}\rt X^{(n+1)}$ is a pushout of boundary inclusions $\dau\cat{R}[r]\rt \cat{R}[r]$ for objects $r\in \cat{R}$ of degree $n+1$ (see Ch.~ 8 of \cite{cis06}). 
Lemma \ref{lem:degeneraciesdeterminedbyboundary} (which appears Lemma 2.6 in \cite{cis14} for strict Eilenberg-Zilber categories) shows that two degenerate elements of a normal presheaf X are the same once their boundaries agree. Instead of the specific cylinder defined in Section \ref{sec:cylinders}, one can  use the cylinder induced by the so-called Lawvere interval (\cite{cis06}, 1.3.9). For any object $X\in \widehat{\cat{R}}$, this cylinder is given by the product $X\times \Omega$ with the subobject classifier in $\widehat{\cat{R}}$.

The arguments in Section \ref{sec:proof} now show that Theorem \ref{thm:existenceofminimalmodels} (a) holds whenever the model structure on dendroidal sets is replaced by 
\setlength{\leftmargini}{1em}
\begin{itemize}
\item[] \emph{any model structure on a category of presheaves over an Eilenberg-Zilber category, in which the cofibrations are the normal monomorphisms.}
\end{itemize}
Apart from the model structure on dendroidal sets, there are many common model categories which are of this form. Examples include the model structure on the category $\widehat{\Lambda}$ of cyclic sets from \cite{dwy85} and the model structure on the category $\widehat{\text{Fin}}$ of symmetric simplicial sets from \cite{ros03} (see also the erratum \cite{ros08}). One can produce many more examples by taking the category of simplicial presheaves over an Eilenberg-Zilber category and equipping it with the generalised Reedy model structure \cite{ber08} or any Bousfield localization thereof. This includes the model structures on $\Gamma$-spaces and (dendroidal) Segal spaces as important examples.

Part (b) of Theorem \ref{thm:existenceofminimalmodels} does not hold in general: it crucially relies on the fact that the category $\Omega/X$ is a \emph{strict} Reedy category when $X$ is normal, in which case it is just a special case of Proposition 2.8 of \cite{cis14}. This property is not shared, for example, by the category $\Gamma$ (or $\Gamma\times\Delta$).

Finally, we would like to illustrate the use of the theory of minimal fibrations by means of the following application:
\begin{example}
Let us consider the following model for the homotopy theory of connective spectra parametrized by a simplicial set $S$. The category $\big(\sSet/S\big)^{\Gamma^{\text{op}}}$ of $\Gamma$-objects in $\sSet/S$ carries a model structure in which 
\begin{itemize}
 \item the cofibrations are the normal monomorphisms
 \item an object $X\colon \Gamma^{\text{op}}\rt \sSet/S$ is a fibrant object if it is Reedy fibrant (with respect to the Kan-Quillen model structure on $\sSet/S$) and the Segal maps
$$\xymatrix{
X(n)\ar[r] & X(1)\times_S X(1)\times_S \cdots \times_S X(1)
}$$
are trivial fibrations for all $n\geq 0$, as is the shear map $X(2)\rt X(1)\times_S X(1)$. 
\end{itemize}
One can easily adapt the classical proof in simplicial sets (see e.~g.~ \cite{gab67}) to prove that any minimal object in this model structure is a locally trivial bundle of $\Gamma$-spaces, i.~e.~ a $\Gamma$-space over $S$ whose pullback to a simplex $\Delta[n]$ is of the form $\Delta[n]\times F$, for some $\Gamma$-space $F$.
\end{example}
Each map of simplicial sets $S\rt S'$ induces a Quillen pair
$$\xymatrix{
\big(\sSet/S\big)^{\Gamma^{\text{op}}} \ar@<1ex>[r] & \big(\sSet/S'\big)^{\Gamma^{\text{op}}}\ar@<1ex>[l]
}$$
where the right adjoint pulls back a $\Gamma$-space over $S'$ to a $\Gamma$-space over $S$. Associating to each simplicial set $S$ the above model category therefore provides a (relative) functor
\begin{equation}\label{eq:gammaspaces}\vcenter{\entrymodifiers={+!!<0pt,\fontdimen22\textfont2>}\xymatrix{
\sSet^{\text{op}}\ar[r] & \cat{ModelCat}^{\text{R}} & S\ar@{|->}[r] & \big(\sSet/S\big)^{\Gamma^{\text{op}}}
}}\end{equation}
A $\Gamma$-space over $S$ is fibrant precisely when its pullback to each simplex of $S$ is fibrant. We can therefore apply the proof of Proposition \ref{prop:gluingleftfibrations} to obtain the following variant of Proposition \ref{prop:descentforleftfibrations} for parametrized $\Gamma$-spaces:
\begin{corollary}
The relative functor \eqref{eq:gammaspaces} preserves homotopy pullbacks.
\end{corollary}
For the same reason, this corollary remains true if one uses the covariant model structure on $\sSet/S$ rather than the Kan-Quillen model structure.

\bibliographystyle{abbrv}
\bibliography{bibliography_minimal_models}

\begin{thebibliography}{10}

\bibitem{bar59}
M.~G. Barratt, V.~K. A.~M. Gugenheim, and J.~C. Moore.
\newblock On semisimplicial fibre-bundles.
\newblock {\em Amer. J. Math.}, 81:639--657, 1959.

\bibitem{bar12}
C.~Barwick.
\newblock On left and right model categories and left and right {B}ousfield
  localizations.
\newblock {\em Homology, Homotopy Appl.}, 12(2):245--320, 2010.

\bibitem{bas12}
M.~Ba{\v{s}}i{\'c} and T.~Nikolaus.
\newblock Dendroidal sets as models for connective spectra.
\newblock {\em J. K-Theory}, 14(3):387--421, 2014.

\bibitem{ber08}
C.~Berger and I.~Moerdijk.
\newblock On an extension of the notion of {R}eedy category.
\newblock {\em Math. Z.}, 269(3-4):977--1004, 2011.

\bibitem{cis06}
D.-C. Cisinski.
\newblock Les pr\'efaisceaux comme mod\`eles des types d'homotopie.
\newblock {\em Ast\'erisque}, (308):xxiv+390, 2006.

\bibitem{cis14}
D.-C. {Cisinski}.
\newblock {Univalent universes for elegant models of homotopy types}.
\newblock {\em \href{http://arxiv.org/abs/1406.0058}{arXiv:1406.0058}}, 2014.

\bibitem{cis09}
D.-C. Cisinski and I.~Moerdijk.
\newblock Dendroidal sets as models for homotopy operads.
\newblock {\em J. Topol.}, 4(2):257--299, 2011.

\bibitem{con83}
A.~Connes.
\newblock Cohomologie cyclique et foncteurs {${\rm Ext}^n$}.
\newblock {\em C. R. Acad. Sci. Paris S\'er. I Math.}, 296(23):953--958, 1983.

\bibitem{dwy85}
W.~G. Dwyer, M.~J. Hopkins, and D.~M. Kan.
\newblock The homotopy theory of cyclic sets.
\newblock {\em Trans. Amer. Math. Soc.}, 291(1):281--289, 1985.

\bibitem{gab67}
P.~Gabriel and M.~Zisman.
\newblock {\em Calculus of fractions and homotopy theory}.
\newblock Ergebnisse der Mathematik und ihrer Grenzgebiete, Band 35.
  Springer-Verlag New York, Inc., New York, 1967.

\bibitem{heu12}
G.~{Heuts}.
\newblock {Algebras over infinity-operads}.
\newblock {\em \href{http://arxiv.org/abs/1110.1776}{arXiv:1110.1776}}, 2011.

\bibitem{joy08}
A.~Joyal.
\newblock The theory of quasi-categories and its applications.
\newblock {\em Preprint}, 2008.

\bibitem{lur09}
J.~Lurie.
\newblock {\em Higher topos theory}, volume 170 of {\em Annals of Mathematics
  Studies}.
\newblock Princeton University Press, Princeton, NJ, 2009.

\bibitem{may67}
J.~P. May.
\newblock {\em Simplicial objects in algebraic topology}.
\newblock Van Nostrand Mathematical Studies, No. 11. D. Van Nostrand Co., Inc.,
  Princeton, N.J.-Toronto, Ont.-London, 1967.

\bibitem{moe10}
I.~Moerdijk.
\newblock Lectures on dendroidal sets.
\newblock In {\em Simplicial methods for operads and algebraic geometry}, Adv.
  Courses Math. CRM Barcelona, pages 1--118. Birkh\"auser/Springer Basel AG,
  Basel, 2010.
\newblock Notes written by J. Guti{\'e}rrez.

\bibitem{moe07}
I.~Moerdijk and I.~Weiss.
\newblock Dendroidal sets.
\newblock {\em Algebr. Geom. Topol.}, 7:1441--1470, 2007.

\bibitem{moe09}
I.~Moerdijk and I.~Weiss.
\newblock On inner {K}an complexes in the category of dendroidal sets.
\newblock {\em Adv. Math.}, 221(2):343--389, 2009.

\bibitem{qui68}
D.~G. Quillen.
\newblock The geometric realization of a {K}an fibration is a {S}erre
  fibration.
\newblock {\em Proc. Amer. Math. Soc.}, 19:1499--1500, 1968.

\bibitem{ros03}
J.~Rosick{\'y} and W.~Tholen.
\newblock Left-determined model categories and universal homotopy theories.
\newblock {\em Trans. Amer. Math. Soc.}, 355(9):3611--3623, 2003.

\bibitem{ros08}
J.~Rosick{\'y} and W.~Tholen.
\newblock Erratum to: ``{L}eft-determined model categories and universal
  homotopy theories'' [{T}rans. {A}mer. {M}ath. {S}oc. 355(9):3611--3623,
  2003].
\newblock {\em Trans. Amer. Math. Soc.}, 360(11):6179, 2008.

\bibitem{seg74}
G.~Segal.
\newblock Categories and cohomology theories.
\newblock {\em Topology}, 13:293--312, 1974.

\bibitem{ber13}
B.~van~den Berg and I.~Moerdijk.
\newblock W-types in homotopy type theory.
\newblock {\em Math. Structures Comput. Sci.}, 25(5):1100--1115, 2015.

\end{thebibliography}


\end{document}